\documentclass[reqno]{amsart}

\usepackage{a4wide}
\usepackage{color}
\usepackage{mathrsfs}
\usepackage{mathtools}
\usepackage{tikz-cd}
\usepackage{amsmath}
\usepackage{amssymb}
\usepackage{bbm}
\numberwithin{equation}{section}
\usepackage[colorlinks,citecolor=green,linkcolor=red]{hyperref}
\usepackage{hyphenat}
\usepackage{enumerate}
\usepackage{esint}

\usepackage[latin1]{inputenc}

\newcommand{\N}{\mathbb{N}}
\newcommand{\R}{\mathbb{R}}
\newcommand{\sfd}{{\sf d}}
\renewcommand{\d}{{\mathrm d}}

\newcommand{\restr}[1]{\lower3pt\hbox{\(|_{#1}\)}}

\newcommand{\nchi}{{\raise.3ex\hbox{\(\chi\)}}}
\newcommand{\Der}{{\rm Der}}
\newcommand{\1}{\mathbbm 1}
\newcommand{\fr}{\penalty-20\null\hfill\(\blacksquare\)}
\newcommand{\mm}{\mathfrak{m}}
\newcommand{\X}{{\rm X}}
\newcommand{\Y}{{\rm Y}}

\newcommand{\LIP}{{\rm LIP}}
\newcommand{\Lip}{{\rm Lip}}
\newcommand{\lip}{{\rm lip}}
\renewcommand{\div}{{\rm div}}

\newtheorem{theorem}{Theorem}[section]
\newtheorem{corollary}[theorem]{Corollary}
\newtheorem{lemma}[theorem]{Lemma}
\newtheorem{proposition}[theorem]{Proposition}
\newtheorem{definition}[theorem]{Definition}
\newtheorem{example}[theorem]{Example}
\newtheorem{remark}[theorem]{Remark}

\title{Preduals of metric BV spaces}

\author{Enrico Pasqualetto}
\address{Department of Mathematics and Statistics,
P.O.\ Box 35 (MaD), FI-40014 University of Jyvaskyla}
\email{enrico.e.pasqualetto@jyu.fi}

\begin{document}
\date{\today} 
\begin{abstract}
We study the predual of the space of functions of bounded variation defined over a metric measure space
$({\rm X},{\sf d},\mathfrak m)$ with $\mathfrak m$ finite. More specifically, for any exponent $p\in(1,\infty)$
we construct an isometric predual of the space ${\rm BV}_p({\rm X})$ of $p$-integrable functions of bounded variation,
which we equip with the norm $\|f\|_{{\rm BV}_p({\rm X})}:=\|f\|_{L^p({\rm X})}+|Df|({\rm X})$. Moreover, we prove that
the standard BV space ${\rm BV}({\rm X}):={\rm BV}_1({\rm X})$, which fails to have a predual for some choices of the
metric measure space, does have a predual in the case where $({\rm X},{\sf d},\mathfrak m)$ is a PI space (i.e.\ a
doubling metric measure space supporting a weak $(1,1)$-Poincar\'{e} inequality) of finite diameter. Along the way,
we also develop a basic theory of BV functions in the setting of extended metric-topological measure spaces, which
is of independent interest.
\end{abstract}
\maketitle
\tableofcontents
\section{Introduction}
In the framework of metric measure spaces, functions of bounded variation (BV functions, for short) and sets of finite perimeter
have been investigated thoroughly, starting from \cite{Mir:03}; see e.g.\ the works
\cite{Amb:01,Amb:02,Amb:DiMa:14,DiMaPhD:14,Martio16,Kin:Kor:Sha:Tuo:14,Lah:20,ambrosio2018rigidity,bru2019rectifiability,Ant:Bre:Pas:24,PasSod25}.
The aim of this paper is to
study a functional-analytic property of metric BV spaces: the existence of an (isometric) predual, i.e.\ of a Banach space
whose continuous dual is (isometrically) isomorphic to the Banach space of BV functions. The presence of a predual of the metric
BV space provides a notion of weak\(^*\) convergence of BV functions, along with a corresponding compactness result (due to the Banach--Alaoglu theorem).
\medskip

In this paper, by a metric measure space \({\bf X}=(\X,\sfd,\mm)\) we mean a metric space \((\X,\sfd)\) equipped with a finite Radon measure
\(\mm\), see Definition \ref{def:mms}; the finiteness assumption is essential for our proof strategy to work. The notion of metric BV space
\({\rm BV}({\bf X})\) that fits nicely our purpose is the one introduced in \cite{DiMar:14,DiMaPhD:14} via an integration-by-parts formula,
in duality with a suitable class of derivations with divergence; it follows from \cite{DiMar:14,DiMaPhD:14}
that this notion of BV is fully equivalent to other ones that were previously studied, such as the approach via Lipschitz
approximation \cite{Mir:03} and the `curvewise' one in terms of test plans \cite{Amb:DiMa:14}.
Roughly speaking, a function \(f\in{\rm BV}({\bf X})\) is canonically associated
with a bounded linear operator \(b\mapsto{\bf L}_f(b)\) that maps each derivation with divergence \(b\in{\rm Der}^\infty_\infty({\bf X})\)
(cf.\ Definitions \ref{def:der} and \ref{def:div}) to a signed Radon measure \({\bf L}_f(b)\) on \(\X\). The latter, which is required
to satisfy the integration-by-parts formula \({\bf L}_f(b)(\X)=-\int f\,\div(b)\,\d\mm\), can be thought of as the `distributional derivative
of \(f\) in the direction of \(b\)'; see Definition \ref{def:BV_der}.
The total variation \({\rm V}(f)\) of \(f\in{\rm BV}({\bf X})\) is defined as the operator norm of \({\bf L}_f\). Letting
\[
\|f\|_{{\rm BV}({\bf X})}\coloneqq\|f\|_{L^1(\mm)}+{\rm V}(f)\quad\text{ for every }f\in{\rm BV}({\bf X}),
\]
we obtain a norm \(\|\cdot\|_{{\rm BV}_p({\bf X})}\) on \({\rm BV}({\bf X})\), which makes it a Banach space. We will show that
\begin{equation}\label{eq:intro_1}
{\rm BV}({\bf X})\text{ can fail to have a predual.}
\end{equation}
Indeed, in Example \ref{ex:no_predual} we construct a (compact) metric measure space whose BV space coincides with \(L^1(\mm)\)
(i.e.\ each integrable function is BV, with null total variation), which -- as it is well known -- fails to have a predual.
On the contrary, our first main result states that -- under mild assumptions on \((\X,\sfd)\) -- for any exponent \(p\in(1,\infty)\) the Banach space
\begin{equation}\label{eq:intro_2}
{\rm BV}_p({\bf X})\text{ has an isometric predual,}
\end{equation}
where (cf.\ Definition \ref{def:BV_der_p}) we define \({\rm BV}_p({\bf X})\coloneqq L^p(\mm)\cap{\rm BV}({\bf X})\)
and we equip it with the norm
\[
\|f\|_{{\rm BV}_p({\bf X})}\coloneqq\|f\|_{L^p(\mm)}+{\rm V}(f);
\]
see Corollary \ref{cor:predual_BV_mms}. In several metric measure spaces of interest, the validity of an embedding theorem
guarantees a higher integrability of functions in \({\rm BV}({\bf X})\), meaning that every \(f\in{\rm BV}({\bf X})\)
belongs to \(L^p(\mm)\) for some exponent \(p\in(1,\infty)\) depending only on the space \({\bf X}\). In other words,
we have that \({\rm BV}({\bf X})={\rm BV}_p({\bf X})\) with equivalent norms, thus accordingly \eqref{eq:intro_2} implies
that \({\rm BV}({\bf X})\) has a predual. For instance, assuming \({\bf X}\) is a PI space (i.e.\ a doubling metric measure
space supporting a weak \((1,1)\)-Poincar\'{e} inequality \cite{HKST:15}) whose diameter is finite,
\({\rm BV}({\bf X})\) and \({\rm BV}_{\alpha_*}({\bf X})\) are isomorphic as Banach spaces, where
\(\alpha_*\coloneqq\frac{\alpha}{\alpha-1}\), for any given \(\alpha>1\) that is greater than a quantity depending only
on the doubling constant of \({\bf X}\). In summary, we obtain that
\begin{equation}\label{eq:intro_3}
{\bf X}=(\X,\sfd,\mm)\text{ is a bounded PI space}\quad\Longrightarrow\quad{\rm BV}({\bf X})\text{ has a predual.}
\end{equation}
See Section \ref{s:predual_BV_PI} for the details. Note that \eqref{eq:intro_1} and \eqref{eq:intro_3} are not in
contradiction, because the metric measure space that is described in Example \ref{ex:no_predual} is not a PI space.
\medskip

Let us now spend a few words on our construction of an isometric predual of \({\rm BV}_p({\bf X})\),
which has been inspired in broad terms by the arguments in \cite[Remark 3.12]{AmbFusPal00}, where it
is shown that the BV space on the Euclidean space has a predual. First, let us make
the additional assumption that \((\X,\sfd)\) is locally compact. The operator \(f\mapsto(f,{\bf L}_f)\) is a linear isometry from
\({\rm BV}_p({\bf X})\) to the product space \(L^p(\mm)\times\mathcal L({\rm Der}^\infty_\infty({\bf X});\mathcal M(\X))\) equipped
with the \(1\)-norm of its component spaces, where \(\mathcal L({\rm Der}^\infty_\infty({\bf X});\mathcal M(\X))\) denotes the
space of bounded linear operators from \({\rm Der}^\infty_\infty({\bf X})\) to the space \(\mathcal M(\X)\) of signed Radon
measures on \(\X\). As \(\X\) is locally compact, the Riesz--Markov--Kakutani theorem guarantees that
\(L^p(\mm)\times\mathcal L({\rm Der}^\infty_\infty({\bf X});\mathcal M(\X))\) is isometrically isomorphic to the dual of
\[
\mathbb W_q({\bf X})\coloneqq L^q(\mm)\times({\rm Der}^\infty_\infty({\bf X})\hat\otimes_\pi C_0(\X))
\]
equipped with the \(\infty\)-norm of its component spaces, where \(q\in(1,\infty)\) is the conjugate exponent to \(p\), while
\(\hat\otimes_\pi\) denotes the projective tensor product. All in all, we have a linear isometry
\[
\phi_{p,{\bf X}}\colon{\rm BV}_p({\bf X})\to\mathbb W_q({\bf X})^*.
\]
The last step is to show that the image of \(\phi_{p,{\bf X}}\) is the annihilator of a vector subspace
\(\mathbb V_q({\bf X})\) of \(\mathbb W_q({\bf X})\), whose definition involves the divergence \(b\mapsto\div(b)\). As a consequence, we have that
\[
{\rm BV}_p({\bf X})\text{ is isometrically isomorphic to }(\mathbb W_q({\bf X})/\mathbb V_q({\bf X}))^*,
\]
in particular the Banach space \({\rm BV}_p({\bf X})\) has an isometric predual. See Theorem \ref{thm:predual_BV_p} for the details.
Local compactness is necessary for the above proof, as it relies on the Riesz--Markov--Kakutani theorem.
In order to remove the local compactness assumption, we employ the following strategy:
\begin{itemize}
\item We extend the definition of \({\rm BV}({\bf Y})\) (and \({\rm BV}_p({\bf Y})\) for \(1<p<\infty\)) to the more general class of extended metric-topological
measure spaces \({\bf Y}=(\Y,\sigma,\rho,\mathfrak n)\) introduced in \cite{AmbrosioErbarSavare16,Sav:22},
by using the theory of Lipschitz derivations that has been recently developed in \cite{PasTai25}.
\item If \((\Y,\sigma)\) is locally compact, then the very same arguments we sketched above show that \({\rm BV}_p({\bf Y})\) has an isometric predual
(in fact, Theorem \ref{thm:predual_BV_p} is formulated for extended spaces).
\item Now, fix a metric measure space \({\bf X}\). We denote by \(\hat{\bf X}=(\hat\X,\hat\tau,\hat\sfd,\hat\mm)\) its Gelfand compactification
\cite[Section 2.1.7]{Sav:22}, which is an extended metric-topological measure space -- whose topology is compact -- wherein \({\bf X}\) can be embedded canonically.
Our goal is to show that
\begin{equation}\label{eq:intro_4}
{\rm BV}_p({\bf X})\text{ and }{\rm BV}_p(\hat{\bf X})\text{ can be identified,}
\end{equation}
getting that \({\rm BV}_p({\bf X})\) has an isometric predual (even when \((\X,\sfd)\) is not locally compact).
\item To verify \eqref{eq:intro_4}, we need to consider also another notion of BV in the extended setting. Namely, we introduce the space \({\rm BV}_*({\bf Y})\)
in Definition \ref{def:BV_rel} (and \({\rm BV}_{*,p}({\bf Y})\) for \(1<p<\infty\) in Definition \ref{def:BV_rel_p}), by mimicking the notion of metric BV space
considered in \cite{Mir:03}. Shortly said, an integrable function \(f\) belongs to \({\rm BV}_*({\bf Y})\) if it can be approximated in energy by a sequence of elements
of the algebra \(\LIP_b(\Y,\sigma,\rho)\) of \(\sigma\)-continuous \(\rho\)-Lipschitz functions. By applying or adapting known techniques and results from
\cite{DiMaPhD:14,PasSod25,PasTai25}, we can prove that
\[
{\rm BV}(\hat{\bf X})\subseteq{\rm BV}_*(\hat{\bf X})\subseteq{\rm BV}_*({\bf X})={\rm BV}({\bf X})\subseteq{\rm BV}(\hat{\bf X})
\]
with the corresponding (in)equalities of norms, whence \eqref{eq:intro_4} follows; see Section \ref{s:equiv_BV}.
\end{itemize}
The language of extended metric-topological measure spaces (or e.m.t.m.s., for short) -- which was introduced
in \cite{AmbrosioErbarSavare16} and investigated further in \cite{Sav:22,Amb:Sav:21} -- is well suited to optimal
transport problems and nonsmooth analysis in infinite-dimensional spaces. Indeed, the class of e.m.t.m.s.\ includes,
in addition to `standard' metric measure spaces, also abstract Wiener spaces \cite{Bog:07} and configuration spaces
\cite{AlbeverioKondratievRockner}, among others. BV functions and sets of finite perimeter have been researched
comprehensively on abstract Wiener spaces (see
e.g.\ \cite{fuk2000,hin04int,hin09set,AmbManMirPal,AmbMirPal10,AmbManMirPal2,GolNov,AmbrosioFigalliRuna,Mir:Nov:Pal:15})
and, more recently, also on configuration spaces \cite{Br:Su}. Whereas in this paper we introduce a BV theory for
e.m.t.m.s.\ mostly as a tool for studying BV spaces over metric measure spaces, we strongly believe it would be
interesting to develop it further, but we leave it for a future work.
\medskip

We conclude the introduction with a list of questions that are left unanswered after this paper:
\begin{itemize}
\item Under which conditions is the predual \(\mathbb W_q({\bf X})/\mathbb V_q({\bf X})\) of \({\rm BV}_p({\bf X})\) separable?
\item Is the predual of \({\rm BV}_p({\bf X})\) unique?
\item Is it true that \({\rm BV}({\bf Y})={\rm BV}_*({\bf Y})\), with the same norms, for every e.m.t.m.s.\ \({\bf Y}\)?
\item Does \({\rm BV}_p({\bf Y})\) have a predual for every e.m.t.m.s.\ \({\bf Y}\)?
\end{itemize}
\textbf{Acknowledgements.} This work was supported by the Research Council of Finland grant 362898. I am grateful to Elio Marconi
for the many enlightening discussions and to Veikko Vuolasto for the careful reading of an earlier draft of this manuscript.
\section{Preliminaries}
\subsection{Tools in Banach space theory}
Let us recall some notions and results concerning Banach spaces. Given a Banach space \(\mathbb V=(\mathbb V,\|\cdot\|_{\mathbb V})\),
we denote by \(\mathbb V^*\) its dual Banach space. The duality pairing between \(\omega\in\mathbb V^*\) and \(v\in\mathbb V\) will be
denoted by \(\langle\omega,v\rangle\coloneqq\omega(v)\in\R\). Furthermore:
\begin{enumerate}[(A)]
\item\label{item:predual} We say that two Banach spaces \(\mathbb V\) and \(\mathbb W\) are \textbf{isomorphic},
and we write \(\mathbb V\approx\mathbb W\), if there exists an isomorphism \(\phi\colon\mathbb V\to\mathbb W\)
of Banach spaces (i.e.\ a linear homeomorphism). If \(\mathbb V\approx\mathbb W^*\), then we say that \(\mathbb W\) is a \textbf{predual} of \(\mathbb V\).
\item\label{item:isometric_predual} We say that \(\mathbb V\) and \(\mathbb W\) are \textbf{isometrically isomorphic},
and we write \(\mathbb V\approx_1\mathbb W\), if there exists an isometric isomorphism
\(\phi\colon\mathbb V\to\mathbb W\) of Banach spaces (i.e.\ a linear isometric homeomorphism).
If \(\mathbb V\approx_1\mathbb W^*\), then we say that \(\mathbb W\) is an \textbf{isometric predual} of \(\mathbb V\).
\item\label{item:ell_infty_product} The Banach space \(\mathbb V\times_\infty\mathbb W\)
is defined as the product vector space \(\mathbb V\times\mathbb W\) endowed with the norm
\(\|(v,w)\|_\infty\coloneqq\max\{\|v\|_{\mathbb V},\|w\|_{\mathbb W}\}\). It holds that
\[
(\mathbb V\times_\infty\mathbb W)^*\approx_1\mathbb V^*\times_1\mathbb W^*,
\]
where the Banach space \(\mathbb V^*\times_1\mathbb W^*\) is defined as the product vector space
\(\mathbb V^*\times\mathbb W^*\) endowed with \(\|(\omega,\eta)\|_1\coloneqq\|\omega\|_{\mathbb V^*}+\|\eta\|_{\mathbb W^*}\).
The duality pairing between \(\mathbb V^*\times_1\mathbb W^*\) and \(\mathbb V\times_\infty\mathbb W\) is
\[
(\mathbb V^*\times_1\mathbb W^*)\times(\mathbb V\times_\infty\mathbb W)\ni((\omega,\eta),(v,w))
\mapsto\langle\omega,v\rangle+\langle\eta,w\rangle\in\R.
\]
\item\label{item:bdd_bilin} We denote by \(\mathcal L(\mathbb V;\mathbb W)\) the Banach space of all bounded linear operators
\(T\colon\mathbb V\to\mathbb W\) endowed with the operator norm \(\|\cdot\|_{\mathcal L(\mathbb V;\mathbb W)}\).
Moreover, we denote by \(\mathcal B(\mathbb V,\mathbb W)\) the Banach space of all bounded bilinear operators
\(B\colon\mathbb V\times\mathbb W\to\R\) endowed with the norm
\[
\|B\|_{\mathcal B(\mathbb V,\mathbb W)}\coloneqq\sup\big\{|B(v,w)|\;\big|\;(v,w)\in\mathbb V\times\mathbb W,\,
\|v\|_{\mathbb V},\|w\|_{\mathbb W}\leq 1\big\}.
\]
It holds that \(\mathcal B(\mathbb V,\mathbb W)\approx_1\mathcal L(\mathbb V;\mathbb W^*)\), via the isometric
isomorphism of Banach spaces given by
\[
\mathcal B(\mathbb V,\mathbb W)\ni B\longmapsto\big(\mathbb V\ni v\mapsto B(v,\cdot)\in\mathbb W^*\big)\in\mathcal L(\mathbb V;\mathbb W^*).
\]
\item\label{item:proj_tensor_prod} We denote by \((\mathbb V\hat\otimes_\pi\mathbb W,\|\cdot\|_\pi)\) the \textbf{projective tensor product}
of the Banach spaces \(\mathbb V\) and \(\mathbb W\), see e.g.\ \cite[Section 2]{Ryan02}. We remind that \(\mathbb V\hat\otimes_\pi\mathbb W\)
is the Banach space obtained as the completion of the normed space \(\mathbb V\otimes_\pi\mathbb W=(\mathbb V\otimes\mathbb W,\|\cdot\|_\pi)\),
where \(\mathbb V\otimes\mathbb W\) is the tensor product of \(\mathbb V\) and \(\mathbb W\) in the sense of vector spaces, while \(\|\cdot\|_\pi\)
denotes the \textbf{projective norm}
\[
\|\alpha\|_\pi\coloneqq\inf\bigg\{\sum_{i=1}^n\|v_i\|_{\mathbb V}\|w_i\|_{\mathbb W}\;\bigg|\;n\in\N,\,(v_i)_{i=1}^n\subseteq\mathbb V,\,
(w_i)_{i=1}^n\subseteq\mathbb W,\,\alpha=\sum_{i=1}^n v_i\otimes w_i\bigg\}.
\]
Recall also that \(\mathcal B(\mathbb V,\mathbb W)\approx_1(\mathbb V\hat\otimes_\pi\mathbb W)^*\), via the isometric
isomorphism that associates to any given \(B\in\mathcal B(\mathbb V,\mathbb W)\) the unique element
\(\bar B\in(\mathbb V\hat\otimes_\pi\mathbb W)^*\) satisfying \(\bar B(v\otimes w)=B(v,w)\) for every
\(v\in\mathbb V\) and \(w\in\mathbb W\); see e.g.\ \cite[Theorem 2.9]{Ryan02}.
\item\label{item:annihilator} If \(\mathbb V\) is a subspace of \(\mathbb W\), then we denote by \(\mathbb V^\perp_{\mathbb W}\)
the \textbf{annihilator} of \(\mathbb V\) in \(\mathbb W^*\), i.e.
\[
\mathbb V^\perp_{\mathbb W}\coloneqq
\big\{\omega\in\mathbb W^*\;\big|\;\langle\omega,v\rangle=0\text{ for every }v\in\mathbb V\big\}.
\]
It holds that \(\mathbb V^\perp_{\mathbb W}\approx_1(\mathbb W/\mathbb V)^*\), via the isometric isomorphism of Banach spaces given by
\[
\mathbb V^\perp_{\mathbb W}\ni\omega\longmapsto
\big(\mathbb W/\mathbb V\ni v+\mathbb V\mapsto\langle\omega,v\rangle\in\R\big)\in(\mathbb W/\mathbb V)^*.
\]
\item\label{item:C_0(X)} Let \((\X,\tau)\) be a locally-compact Hausdorff space. We denote by \(C_b(\X,\tau)\) the vector space
of all bounded real-valued continuous functions on \((\X,\tau)\). The space \(C_b(\X,\tau)\) is a Banach space if endowed
with the supremum norm \(\|f\|_{C_b(\X,\tau)}\coloneqq\sup_{x\in\X}|f(x)|\). We then consider its closed linear subspace
\(C_0(\X,\tau)\), which is defined as
\[
C_0(\X,\tau)\coloneqq{\rm cl}_{C_b(\X,\tau)}(C_c(\X,\tau)),
\]
where \(C_c(\X,\tau)\) denotes the set of all those functions in \(C_b(\X,\tau)\) having compact support.
\item\label{item:M(X)} We denote by \(\mathcal M(\X,\tau)\)
the vector space of all finite signed Radon measures on \((\X,\tau)\). The space \(\mathcal M(\X,\tau)\) is a
Banach space if endowed with the total variation norm \(\|\cdot\|_{\rm TV}\). The Riesz--Markov--Kakutani
theorem states that \(\mathcal M(\X,\tau)\approx_1 C_0(\X,\tau)^*\), with duality pairing
\[
\langle\mu,f\rangle\coloneqq\int f\,\d\mu\quad\text{ for every }\mu\in\mathcal M(\X,\tau)\text{ and }f\in C_0(\X,\tau).
\]
\end{enumerate}
We refer e.g.\ to \cite{FHHMZ11} for a thorough treatise on Banach spaces.
\subsection{Extended metric-topological measure spaces}
Let us recall the notion of extended metric-topological measure space, which was first
introduced in \cite{AmbrosioErbarSavare16} and studied further in \cite{Sav:22}:
\begin{definition}[Extended metric-topological measure space]\label{def:emms}
A quadruple \({\bf X}=(\X,\tau,\sfd,\mm)\) is said to be an \textbf{extended metric-topological measure space}
(or an \textbf{e.m.t.m.s.}, for short) provided:
\begin{enumerate}[(i)]
\item \((\X,\tau)\) is a topological space and \((\X,\sfd)\) is an extended metric space.
\item\label{it:emtms_top} The topology \(\tau\) is generated by \(\LIP_b(\X,\tau,\sfd)\coloneqq C_b(\X,\tau)\cap\LIP(\X,\sfd)\).
\item\label{it:emtms_recov} We have
\[
\sfd(x,y)=\sup\big\{|f(x)-f(y)|\;\big|\;f\in\LIP_b(\X,\tau,\sfd),\,\Lip(f)\leq 1\big\}\quad\text{ for every }x,y\in\X.
\]
\item \(\mm\) is a finite, non-negative Radon measure on \((\X,\tau)\).
\end{enumerate}
\end{definition}

Condition \eqref{it:emtms_top} implies that \(\tau\) is the topology generated by \(C_b(\X,\tau)\),
or equivalently that \((\X,\tau)\) is a completely regular space. Condition \eqref{it:emtms_recov}
ensures that \(\tau\) is a Hausdorff topology (thus, \((\X,\tau)\) is a Tychonoff space) and that
\(\sfd\colon\X\times\X\to[0,+\infty]\) is a \(\tau\otimes\tau\)-lower semicontinuous function.
\medskip

We now collect some properties of \(\LIP_b(\X,\tau,\sfd)\). It follows from \cite[Lemma 2.1.27]{Sav:22} that 
\begin{equation}\label{eq:LIP_dense_L1}
\LIP_b(\X,\tau,\sfd)\quad\text{ is dense in }L^1(\mm).
\end{equation}
Moreover, \(\LIP_b(\X,\tau,\sfd)\) is a Banach algebra with respect to the pointwise operations and the norm
\[
\|f\|_{\LIP_b(\X,\tau,\sfd)}\coloneqq\|f\|_{C_b(\X,\tau)}+\Lip(f),
\]
as it was observed in \cite[Remark 2.1]{PasTai25}.
To any \(f\in\LIP_b(\X,\tau,\sfd)\), we associate the \textbf{asymptotic Lipschitz constant} function
\(\lip_a(f)\colon\X\to[0,+\infty)\), which is defined as
\[
\lip_a(f)(x)\coloneqq\inf\big\{\Lip(f|_U)\;\big|\;x\in U\in\tau\big\}\quad\text{ for every }x\in\X.
\]
We have that \(\lip_a(f)\colon\X\to[0,+\infty)\) is \(\tau\)-lower semicontinuous.
\begin{definition}[Gelfand compactification of an e.m.t.m.s.]
Let \({\bf X}=(\X,\tau,\sfd,\mm)\) be an e.m.t.m.s. Then we denote by \(\hat{\bf X}=(\hat\X,\hat\tau,\hat\sfd,\hat\mm)\) the
\textbf{Gelfand compactification} of \({\bf X}\) and by \(\iota\colon\X\hookrightarrow\hat\X\) the associated embedding map,
as in \cite[Section 2.1.7]{Sav:22}.
\end{definition}

Below we collect the properties of the Gelfand compactification \(\hat{\bf X}\) of \({\bf X}\), cf.\ \cite[Theorem 2.1.34]{Sav:22}:
\begin{itemize}
\item \(\hat{\bf X}=(\hat\X,\hat\tau,\hat\sfd,\hat\mm)\) is an e.m.t.m.s.\ with \((\hat\X,\hat\tau)\) compact,
\item \(\iota\colon(\X,\tau)\to(\hat\X,\hat\tau)\) is a homeomorphism onto its image and \(\iota(\X)\) is dense
in \((\hat\X,\hat\tau)\),
\item \(\hat\sfd(\iota(x),\iota(y))=\sfd(x,y)\) for every \(x,y\in\X\),
\item \(\hat\mm=\iota_\#\mm\), i.e.\ \(\hat\mm\) is the pushforward measure of \(\mm\) under \(\iota\),
\item the map \(\iota^*\colon L^1(\hat\mm)\to L^1(\mm)\) (given by Remark \ref{rmk:compos_bdd_compr} below) is an isomorphism of Banach
spaces and of Riesz spaces, whose inverse we denote by \(\iota_*\colon L^1(\mm)\to L^1(\hat\mm)\).
\end{itemize}
\begin{remark}\label{rmk:compos_bdd_compr}{\rm
Let \((\X,\Sigma,\mm)\), \((\Y,\Omega,\mathfrak n)\) be finite measure spaces. Let \(\phi\colon(\X,\Sigma)\to(\Y,\Omega)\)
be a measurable map with \(\phi_\#\mm\leq C\mathfrak n\) for some \(C>0\). Then \(\phi\) induces, by pre-composition, a map
\[
\phi^*\colon L^1(\mathfrak n)\to L^1(\mm).
\]
Namely, \(\phi^*\) maps the \(\mathfrak n\)-a.e.\ equivalence class of any given \(\mathfrak n\)-integrable
function \(\bar f\colon\Y\to\R\) to the \(\mm\)-a.e.\ equivalence class of \(\bar f\circ\phi\). Clearly,
the map \(\phi^*\) is linear and \(\|\phi^*f\|_{L^1(\mm)}\leq C\|f\|_{L^1(\mathfrak n)}\) for every \(f\in L^1(\mathfrak n)\).
\fr}\end{remark}

It was proved in \cite[Lemma 2.13]{PasTai25} that, for any given e.m.t.m.s.\ \({\bf X}=(\X,\tau,\sfd,\mm)\), the operator
\[
\LIP_b(\hat\X,\hat\tau,\hat\sfd)\ni\hat f\mapsto\hat f\circ\iota\in\LIP_b(\X,\tau,\sfd)
\]
is an isomorphism of Banach algebras. We denote it by \(\iota^*\colon\LIP_b(\hat\X,\hat\tau,\hat\sfd)\to\LIP_b(\X,\tau,\sfd)\),
while its inverse will be denoted by \(\iota_*\colon\LIP_b(\X,\tau,\sfd)\to\LIP_b(\hat\X,\hat\tau,\hat\sfd)\).
\medskip

In this paper, our interest is mostly focussed on the distinguished subclass of metric measure spaces,
by which we mean those e.m.t.m.s.\ \((\X,\tau,\sfd,\mm)\) where \(\sfd\) is a (non-extended) distance and \(\tau\)
is exactly the topology induced by \(\sfd\). Namely:
\begin{definition}[Metric measure space]\label{def:mms}
We say that a triple \({\bf X}=(\X,\sfd,\mm)\) is a \textbf{metric measure space} (or an \textbf{m.m.s.}, for short)
provided \((\X,\sfd)\) is a metric space and \(\mm\) is a finite, non-negative Radon measure on \((\X,\sfd)\).
\end{definition}
\subsection{Lipschitz derivations}
Let us recall the concept of Lipschitz derivation for e.m.t.m.s., which was introduced in
\cite{PasTai25} as a generalisation of the definition given in \cite{DiMar:14,DiMaPhD:14} (see also \cite{Weaver2018}).
\begin{definition}[Derivation]\label{def:der}
Let \({\bf X}=(\X,\tau,\sfd,\mm)\) be an e.m.t.m.s. Then we say that a linear operator
\(b\colon\LIP_b(\X,\tau,\sfd)\to L^\infty(\mm)\) is a \textbf{derivation} on \({\bf X}\) if the following conditions hold:
\begin{enumerate}[(i)]
\item \textsc{Leibniz rule.} We have \(b(fg)=f\,b(g)+g\,b(f)\) for every \(f,g\in\LIP_b(\X,\tau,\sfd)\).
\item \textsc{Continuity property.} There exists a function \(G\in L^\infty(\mm)^+\) such that
\begin{equation}\label{eq:cont_der}
|b(f)|\leq G\,\lip_a(f)\quad\text{ in the }\mm\text{-a.e.\ sense, for every }f\in\LIP_b(\X,\tau,\sfd).
\end{equation}
\end{enumerate}
We denote by \({\rm Der}^\infty({\bf X})\) the collection of all derivations on \(\bf X\).
\end{definition}

Note that \({\rm Der}^\infty({\bf X})\) is a module over the ring \(L^\infty(\mm)\) (thus in particular a vector space) if
endowed with the usual pointwise operations. Given any \(b\in{\rm Der}^\infty({\bf X})\), the set of all \(G\in L^\infty(\mm)^+\)
satisfying \eqref{eq:cont_der} is a closed sublattice of \(L^\infty(\mm)^+\), thus it admits a unique \(\mm\)-a.e.\ minimal
element \(|b|\in L^\infty(\mm)^+\). Moreover, \({\rm Der}^\infty({\bf X})\) is a Banach space with respect to the norm
\[
\|b\|_{{\rm Der}^\infty({\bf X})}\coloneqq\||b|\|_{L^\infty(\mm)}\quad\text{ for every }b\in{\rm Der}^\infty({\bf X}).
\]
\begin{definition}[Divergence]\label{def:div}
Let \({\bf X}=(\X,\tau,\sfd,\mm)\) be an e.m.t.m.s. Then we say that a derivation \(b\in{\rm Der}^\infty({\bf X})\)
\textbf{has divergence} if there exists a function \(\div(b)\in L^\infty(\mm)\) such that
\[
\int b(f)\,\d\mm=-\int f\,\div(b)\,\d\mm\quad\text{ for every }f\in\LIP_b(\X,\tau,\sfd).
\]
We denote by \({\rm Der}^\infty_\infty({\bf X})\) the collection of all derivations on \({\bf X}\) having divergence.
\end{definition}

Note that \(\div(b)\) is uniquely determined by the density of \(\LIP_b(\X,\tau,\sfd)\) in \(L^1(\mm)\), cf.\ \eqref{eq:LIP_dense_L1}.
It can be readily checked that \({\rm Der}^\infty_\infty({\bf X})\) is a vector subspace of
\({\rm Der}^\infty({\bf X})\) and \(\div\colon{\rm Der}^\infty_\infty({\bf X})\to L^\infty(\mm)\) is
a linear operator. In particular, we have that
\[
\overline{\rm Der}^\infty_\infty({\bf X})\coloneqq{\rm cl}_{{\rm Der}^\infty({\bf X})}({\rm Der}^\infty_\infty({\bf X}))
\]
is a Banach space with respect to the norm induced by \({\rm Der}^\infty({\bf X})\).
\medskip

Under suitable topological assumptions, the space \({\rm Der}^\infty_\infty({\bf X})\) can be identified with
\({\rm Der}^\infty_\infty(\hat{\bf X})\):
\begin{proposition}\label{prop:der_on_Gelfand_cpt}
Let \({\bf X}=(\X,\tau,\sfd,\mm)\) be an e.m.t.m.s.\ such that the topology \(\tau\) is metrisable on each \(\tau\)-compact subset
of \(\X\). Let us define the operator \(\iota_*\colon{\rm Der}^\infty_\infty({\bf X})\to{\rm Der}^\infty_\infty(\hat{\bf X})\) as
\[
(\iota_*b)(\hat f)\coloneqq\iota_*(b(\iota^*\hat f))\in L^\infty(\hat\mm)\quad\text{ for every }b\in{\rm Der}^\infty_\infty({\bf X})
\text{ and }\hat f\in\LIP_b(\hat\X,\hat\tau,\hat\sfd).
\]
Then we have that \(\iota_*\colon{\rm Der}^\infty_\infty({\bf X})\to{\rm Der}^\infty_\infty(\hat{\bf X})\) is a linear isomorphism,
whose inverse we denote by \(\iota^*\colon{\rm Der}^\infty_\infty(\hat{\bf X})\to{\rm Der}^\infty_\infty({\bf X})\).
Moreover, given any \(b\in{\rm Der}^\infty_\infty({\bf X})\) and \(h\in L^\infty(\mm)\), we have
\[
\iota_*(hb)=(\iota_*h)(\iota_*b),\qquad|\iota_*b|=\iota_*|b|,\qquad\div(\iota_*b)=\iota_*(\div(b)).
\]
\end{proposition}

The above result was proved in \cite[Proposition 4.17]{PasTai25}.
\section{Functions of bounded variation on extended spaces}
\subsection{Definition by means of derivations}\label{s:BV_der}
We propose the following definition of BV function on e.m.t.m.s., which is a natural
extension of the one for m.m.s.\ that was introduced in \cite{DiMar:14,DiMaPhD:14}:
\begin{definition}[Function of bounded variation]\label{def:BV_der}
Let \({\bf X}=(\X,\tau,\sfd,\mm)\) be an e.m.t.m.s. Let \(f\in L^1(\mm)\) be given.
Then we say that \(f\) is a \textbf{function of bounded variation} provided there exists a bounded linear
operator \({\bf L}_f\colon({\rm Der}^\infty_\infty({\bf X}),\|\cdot\|_{{\rm Der}^\infty({\bf X})})\to\mathcal M(\X,\tau)\) such that
\begin{subequations}\begin{align}
\label{eq:def_BV_der_1}
{\bf L}_f(b)(\X)=-\int f\,\div(b)\,\d\mm&\quad\text{ for every }b\in{\rm Der}^\infty_\infty({\bf X}),\\
\label{eq:def_BV_der_2}
{\bf L}_f(hb)=h\,{\bf L}_f(b)&\quad\text{ for every }h\in\LIP_b(\X,\tau,\sfd)\text{ and }b\in{\rm Der}^\infty_\infty({\bf X}).
\end{align}\end{subequations}
We denote by \({\rm BV}({\bf X})\) the space of all functions of bounded variation on \({\bf X}\).
\end{definition}

For any \(f\in{\rm BV}({\bf X})\), the map \({\bf L}_f\) is uniquely determined: if \(\tilde{\bf L}_f\) has the same properties, then
\[
\int h\,\d{\bf L}_f(b)={\bf L}_f(hb)=-\int f\,\div(hb)\,\d\mm=\tilde{\bf L}_f(hb)=\int h\,\d\tilde{\bf L}_f(b)
\]
for every \(b\in{\rm Der}^\infty_\infty({\bf X})\) and \(h\in\LIP_b(\X,\tau,\sfd)\), so that
\({\bf L}_f(b)=\tilde{\bf L}_f(b)\) by the density of \(\LIP_b(\X,\tau,\sfd)\) in \(L^1(|{\bf L}_f(b)-\tilde{\bf L}_f(b)|)\).
Also, \({\bf L}_f\) can be uniquely extended to an element of
\(\mathcal L\big(\overline{\rm Der}^\infty_\infty({\bf X});\mathcal M(\X,\tau)\big)\), which we still denote by \({\bf L}_f\).
It is easy to check that \({\rm BV}({\bf X})\) is a vector subspace of \(L^1(\mm)\), as well as a Banach space if endowed with the norm
\[
\|f\|_{{\rm BV}({\bf X})}\coloneqq\|f\|_{L^1(\mm)}+{\rm V}(f),
\]
where the \textbf{total variation} \({\rm V}(f)\) of the function \(f\) is defined as
\[
{\rm V}(f)\coloneqq\|{\bf L}_f\|_{\mathcal L(\overline{\rm Der}^\infty_\infty({\bf X});\mathcal M(\X,\tau))}.
\]

We also consider the following variants of the BV space, with a different integrability condition:
\begin{definition}[The space \({\rm BV}_p({\bf X})\)]\label{def:BV_der_p}
Let \({\bf X}=(\X,\tau,\sfd,\mm)\) be an e.m.t.m.s.\ and \(p\in[1,\infty)\). Then we define the vector subspace
\({\rm BV}_p({\bf X})\) of \({\rm BV}({\bf X})\) as
\[
{\rm BV}_p({\bf X})\coloneqq L^p(\mm)\cap{\rm BV}({\bf X}).
\]
Moreover, we endow the space \({\rm BV}_p({\bf X})\) with the norm
\[
\|f\|_{{\rm BV}_p({\bf X})}\coloneqq\|f\|_{L^p(\mm)}+{\rm V}(f).
\]
\end{definition}

Note that \({\rm BV}_1({\bf X})={\rm BV}({\bf X})\). Moreover, it is easy to check that \({\rm BV}_p({\bf X})\)
is a Banach space. To this aim, suppose \((f_n)_n\subseteq{\rm BV}_p({\bf X})\) is Cauchy. In particular,
\((f_n)_n\) is Cauchy in \(L^p(\mm)\), thus there exists \(f\in L^p(\mm)\) such that \(f_n\to f\) in \(L^p(\mm)\).
Moreover, \((f_n)_n\) is Cauchy also in \({\rm BV}({\bf X})\) by H\"{o}lder's inequality, so that there
exists \(\tilde f\in{\rm BV}({\bf X})\) such that \(f_n\to\tilde f\) in \({\rm BV}({\bf X})\). In particular, we have
\(f_n\to\tilde f\) in \(L^1(\mm)\), thus \(f=\tilde f\) and \(f_n\to f\) in \({\rm BV}_p({\bf X})\). This shows
that \({\rm BV}_p({\bf X})\) is complete.
\subsection{Definition by means of Lipschitz approximations}\label{s:BV_rel}
Another notion of BV space for e.m.t.m.s.\ that we propose is the following one,
which generalises its m.m.s.\ version from \cite{Mir:03}:
\begin{definition}[Function of bounded variation in the relaxed sense]\label{def:BV_rel}
Let \({\bf X}=(\X,\tau,\sfd,\mm)\) be an e.m.t.m.s. Let \(f\in L^1(\mm)\) be given. Then we say that \(f\) is a
\textbf{function of bounded variation in the relaxed sense} provided there exists a sequence \((f_n)_n\subseteq\LIP_b(\X,\tau,\sfd)\)
such that
\[
f_n\to f\quad\text{ in }L^1(\mm),\qquad\sup_{n\in\N}\int\lip_a(f_n)\,\d\mm<+\infty.
\]
We denote by \({\rm BV}_*({\bf X})\) the space of all functions of bounded variation in the relaxed sense on \({\bf X}\).
\end{definition}

It is easy to check that \({\rm BV}_*({\bf X})\) is a vector subspace of \(L^1(\mm)\). For any \(f\in{\rm BV}_*({\bf X})\), we define
\[
{\rm V}_*(f)\coloneqq\inf\bigg\{\varliminf_{n\to\infty}\int\lip_a(f_n)\,\d\mm\;\bigg|\;(f_n)_n\subseteq\LIP_b(\X,\tau,\sfd),
\,f_n\to f\text{ in }L^1(\mm)\bigg\}.
\]
Observe that, as an immediate consequence of the definition, we have
\begin{equation}\label{eq:const_BV}
f\in{\rm BV}_*({\bf X}),\quad{\rm V}_*(f)\leq\int\lip_a(f)\,\d\mm\quad\text{ for every }f\in\LIP_b(\X,\tau,\sfd).
\end{equation}
Moreover, we claim that the following chain rule holds:
\begin{equation}\label{eq:chain_rule_BV}
\phi\circ f\in{\rm BV}_*({\bf X}),\quad{\rm V}_*(\phi\circ f)\leq\Lip(\phi){\rm V}_*(f)\quad\text{ for every }
f\in{\rm BV}_*({\bf X})\text{ and }\phi\in\LIP(\R).
\end{equation}
Indeed, if \((f_n)_n\subseteq\LIP_b(\X,\tau,\sfd)\) is chosen so that \(f_n\to f\) in \(L^1(\mm)\) and \(\int\lip_a(f_n)\,\d\mm\to{\rm V}_*(f)\), then
\((\phi\circ f_n)_n\subseteq\LIP_b(\X,\tau,\sfd)\) and \(\int|\phi\circ f_n-\phi\circ f|\,\d\mm\leq\Lip(\phi)\int|f_n-f|\,\d\mm\to 0\), thus accordingly
\[
{\rm V}_*(\phi\circ f)\leq\varliminf_{n\to\infty}\int\lip_a(\phi\circ f_n)\,\d\mm\leq\Lip(\phi)\lim_{n\to\infty}\int\lip_a(f_n)\,\d\mm=\Lip(\phi){\rm V}_*(f).
\]

Also in this case, we consider the following variants of the BV space:
\begin{definition}[The space \({\rm BV}_{*,p}({\bf X})\)]\label{def:BV_rel_p}
Let \({\bf X}=(\X,\tau,\sfd,\mm)\) be an e.m.t.m.s.\ and \(p\in[1,\infty)\). Then we define the vector subspace
\({\rm BV}_{*,p}({\bf X})\) of \({\rm BV}_*({\bf X})\) as
\[
{\rm BV}_{*,p}({\bf X})\coloneqq L^p(\mm)\cap{\rm BV}_*({\bf X}).
\]
Moreover, we endow the space \({\rm BV}_{*,p}({\bf X})\) with the norm
\[
\|f\|_{{\rm BV}_{*,p}({\bf X})}\coloneqq\|f\|_{L^p(\mm)}+{\rm V}_*(f).
\]
\end{definition}

Note that \({\rm BV}_{*,1}({\bf X})={\rm BV}_*({\bf X})\), and that \({\rm BV}_{*,p}({\bf X})\) is a Banach space
for every \(p\in[1,\infty)\).
\subsection{Equivalence statements}\label{s:equiv_BV}
The goal of this section is to prove that if \({\bf X}\) is a metric measure space satisfying rather mild assumptions,
then the BV spaces we discussed in Sections \ref{s:BV_der} and \ref{s:BV_rel} are invariant under passing to the Gelfand compactification
\(\hat{\bf X}\); cf.\ Corollary \ref{cor:equiv_BV_on_Gelfand_cpt}.
\medskip

In the setting of metric measure spaces, we know from \cite[Theorem 4.9]{PasSod25} (after \cite{DiMar:14,DiMaPhD:14}) that
-- under mild assumptions -- the spaces \({\rm BV}({\bf X})\) and \({\rm BV}_*({\bf X})\) coincide:
\begin{theorem}\label{thm:BV=BV*_mms}
Let \({\bf X}=(\X,\sfd,\mm)\) be a m.m.s.\ such that \((\X,\sfd)\) is either complete or Radon. Then we have
\({\rm BV}({\bf X})={\rm BV}_*({\bf X})\) and \({\rm V}(f)={\rm V}_*(f)\) for every \(f\in{\rm BV}({\bf X})\).
\end{theorem}

By a \emph{Radon} metric space we mean a metric space \((\X,\sfd)\) whose induced topological space \((\X,\tau_\sfd)\) is Radon,
meaning that every finite Borel measure on \((\X,\tau_\sfd)\) is a Radon measure. It is known, for instance, that every Borel
subset of a complete and separable metric space is a Radon space.
\medskip

In the general context of e.m.t.m.s., we can still prove the inclusion \({\rm BV}({\bf X})\subseteq{\rm BV}_*({\bf X})\):
\begin{theorem}\label{thm:BV_in_BV*}
Let \({\bf X}=(\X,\tau,\sfd,\mm)\) be an e.m.t.m.s. Then we have \({\rm BV}({\bf X})\subseteq{\rm BV}_*({\bf X})\) and
\[
{\rm V}_*(f)\leq{\rm V}(f)\quad\text{ for every }f\in{\rm BV}({\bf X}).
\]
\end{theorem}
\begin{proof}
The statement can be proved by repeating verbatim the arguments in \cite[Section 4]{PasSod25}. For the reader's usefulness, we give a brief sketch of the proof
(omitting some details):
\begin{enumerate}[(a)]
\item\label{it:min_relax_slope} For any \(f\in\LIP_b(\X,\tau,\sfd)\), there exists a (unique) \(\mm\)-a.e.\ minimal \(|Df|_*\in L^1(\mm)^+\),
called the \emph{minimal relaxed slope} of \(f\), such that there exists \((f_n)_n\subseteq\LIP_b(\X,\tau,\sfd)\)
for which \(f_n\to f\) in \(L^1(\mm)\) and \(\lip_a(f_n)\rightharpoonup|Df|_*\) weakly in \(L^1(\mm)\).
Clearly, we have that \(|Df|_*\leq\lip_a(f)\) \(\mm\)-a.e.
\item\label{it:cotg_mod} There exist an \emph{\(L^1(\mm)\)-normed \(L^1(\mm)\)-module} \(L^1(T^*{\bf X})\) (in the sense of
\cite[Definition 1.2.10]{Gigli14}), which we call the \emph{cotangent module} of \({\bf X}\), and a linear map
\(\d\colon\LIP_b(\X,\tau,\sfd)\to L^1(T^*{\bf X})\) such that \(|\d f|=|Df|_*\) for all \(f\in\LIP_b(\X,\tau,\sfd)\).
Also, \(\d(fg)=f\,\d g+g\,\d f\) for all \(f,g\in\LIP_b(\X,\tau,\sfd)\).
\item\label{it:adj_diff} We can regard \(\d\) as a densely-defined unbounded linear operator \(\d\colon L^1(\mm)\to L^1(T^*{\bf X})\),
whose adjoint we denote by \(\d^*\colon L^1(T^*{\bf X})^*\to L^\infty(\mm)\). Each \(v\in L^\infty(T{\bf X})\coloneqq L^1(T^*{\bf X})^*\)
in the domain of \(\d^*\) induces a derivation \(b_v\in{\rm Der}^\infty_\infty({\bf X})\) satisfying \(|b_v|\leq\|v\|_{L^\infty(T{\bf X})}\) and \(\div(b_v)=-\d^*v\).
\item\label{it:energy} We denote by \(\mathcal E\colon L^1(\mm)\to[0,+\infty]\) the `relaxed BV energy', which is given by \(\mathcal E(f)\coloneqq{\rm V}_*(f)\)
if \(f\in{\rm BV}_*({\bf X})\) and \(\mathcal E(f)\coloneqq+\infty\) otherwise. We also define the functional \(\tilde{\mathcal E}\colon L^1(\mm)\to[0,+\infty]\)
as \(\tilde{\mathcal E}(f)\coloneqq\int|Df|_*\,\d\mm\) if \(f\in\LIP_b(\X,\tau,\sfd)\) and \(\tilde{\mathcal E}(f)\coloneqq+\infty\) otherwise. Since
we have that \({\rm V}_*(f)\leq\int|Df|_*\,\d\mm\leq\int\lip_a(f)\,\d\mm\) holds for every \(f\in\LIP_b(\X,\tau,\sfd)\), we see that \(\mathcal E\) coincides with the
\(L^1(\mm)\)-lower semicontinuous envelope of \(\tilde{\mathcal E}\).
\item\label{it:conclusion} Now, fix \(f\in{\rm BV}({\bf X})\). Since \(\mathcal E\) is convex and \(L^1(\mm)\)-lower semicontinuous, it coincides
with the Fenchel biconjugate \(\tilde{\mathcal E}^{**}\) of \(\tilde{\mathcal E}\). Hence, by applying \eqref{it:adj_diff} and \cite[Theorem 5.1]{BBS} we obtain that
\[\begin{split}
\mathcal E(f)&=\tilde{\mathcal E}^{**}(f)\leq\sup\bigg\{-\int f\,\div(b)\,\d\mm\;\bigg|\;b\in{\rm Der}^\infty_\infty({\bf X}),\,\|b\|_{{\rm Der}^\infty({\bf X})}\leq 1\bigg\}\\
&\leq\sup\big\{|{\bf L}_f(b)(\X)|\;\big|\;b\in{\rm Der}^\infty_\infty({\bf X}),\,\|b\|_{{\rm Der}^\infty({\bf X})}\leq 1\big\}
\leq\|{\bf L}_f\|_{\mathcal L(\overline{\rm Der}^\infty_\infty({\bf X});\mathcal M(\X,\tau))}={\rm V}(f).
\end{split}\]
This gives that \(f\in{\rm BV}_*({\bf X})\) and \({\rm V}_*(f)=\mathcal E(f)\leq{\rm V}(f)\), which is the sought conclusion.
\end{enumerate}
For more details on \eqref{it:cotg_mod}, \eqref{it:adj_diff}, \eqref{it:conclusion}, see \cite[Theorem 3.9, Lemma 4.7, Theorem 4.8]{PasSod25}, respectively.
\end{proof}
\begin{lemma}\label{lem:compos_BV*}
Let \({\bf X}=(\X,\tau,\sfd,\mm)\) and \({\bf Y}=(\Y,\sigma,\rho,\mathfrak n)\) be e.m.t.m.s. Let \(\phi\colon\X\to\Y\)
be a given map such that \(\phi\colon(\X,\tau)\to(\Y,\sigma)\) is continuous, \(\phi\colon(\X,\sfd)\to(\Y,\rho)\) is
Lipschitz and \(\phi_\#\mm\leq C\mathfrak n\) for some \(C>0\). Then, letting \(\phi^*\) be as in Remark \ref{rmk:compos_bdd_compr},
we have \(\phi^*({\rm BV}_*({\bf Y}))\subseteq{\rm BV}_*({\bf X})\) and
\[
{\rm V}_*(\phi^*f)\leq C\Lip(\phi){\rm V}_*(f)\quad\text{ for every }f\in{\rm BV}_*({\bf Y}).
\]
\end{lemma}
\begin{proof}
Fix any \(f\in{\rm BV}_*({\bf Y})\). Take a sequence \((f_n)_n\subseteq\LIP_b(\Y,\sigma,\rho)\) such that \(f_n\to f\)
in \(L^1(\mathfrak n)\) and \(\int\lip_a(f_n)\,\d\mathfrak n\to{\rm V}_*(f)\). Observe that
\(f_n\circ\phi\in\LIP_b(\X,\tau,\sfd)\) for every \(n\in\N\). Moreover, we have
\(\|f_n\circ\phi-\phi^*f\|_{L^1(\mm)}\leq C\|f_n-f\|_{L^1(\mathfrak n)}\to 0\) as \(n\to\infty\) and,
since \(\lip_a(f_n\circ\phi)\leq\Lip(\phi)\lip_a(f_n)\circ\phi\),
\[
\int\lip_a(f_n\circ\phi)\,\d\mm\leq\Lip(\phi)\int\lip_a(f_n)\circ\phi\,\d\mm\leq C\Lip(\phi)\int\lip_a(f_n)\,\d\mathfrak n
\quad\text{ for every }n\in\N.
\]
Letting \(n\to\infty\), we thus conclude that \(\phi^*f\in{\rm BV}_*({\bf X})\) and \({\rm V}_*(\phi^*f)\leq C\Lip(\phi){\rm V}_*(f)\).
\end{proof}
\begin{proposition}\label{prop:compos_BV}
Let \({\bf X}=(\X,\tau,\sfd,\mm)\) be an e.m.t.m.s.\ such that the topology \(\tau\) is metrisable on
each \(\tau\)-compact set. Then, letting \(\iota^*\) be as in Proposition \ref{prop:der_on_Gelfand_cpt}, we have \(\iota_*({\rm BV}({\bf X}))\subseteq{\rm BV}(\hat{\bf X})\) and
\[
{\bf L}_{\iota_*f}(\hat b)=\iota_\#({\bf L}_f(\iota^*\hat b))\in\mathcal M(\hat\X,\hat\tau)\quad
\text{ for every }f\in{\rm BV}({\bf X})\text{ and }\hat b\in{\rm Der}^\infty_\infty(\hat{\bf X}).
\]
Moreover, we have \({\rm V}(\iota_*f)={\rm V}(f)\) for every \(f\in{\rm BV}({\bf X})\).
\end{proposition}
\begin{proof}
Fix any \(f\in{\rm BV}({\bf X})\). Given any \(\hat b\in{\rm Der}^\infty_\infty(\hat{\bf X})\), we define
\({\bf L}(\hat b)\coloneqq\iota_\#({\bf L}_f(\iota^*\hat b))\). Since \({\bf L}_f(\iota^*\hat b)\) is a finite signed
Radon measure on \((\X,\tau)\) and \(\iota\colon(\X,\tau)\to(\hat\X,\hat\tau)\) is continuous, we deduce that
\({\bf L}(\hat b)\) is a finite signed Radon measure on \((\hat\X,\hat\tau)\). The resulting operator
\({\bf L}\colon{\rm Der}^\infty_\infty(\hat{\bf X})\to\mathcal M(\hat\X,\hat\tau)\) is linear by construction.
Using Proposition \ref{prop:der_on_Gelfand_cpt} and the assumption \(f\in{\rm BV}({\bf X})\), we obtain that
\[
{\bf L}(\hat b)(\hat\X)={\bf L}_f(\iota^*\hat b)(\X)=-\int f\,\div(\iota^*\hat b)\,\d\mm
=-\int f\,\iota^*(\div(\hat b))\,\d\mm=-\int(\iota_*f)\,\div(\hat b)\,\d\hat\mm
\]
for every \(\hat b\in{\rm Der}^\infty_\infty(\hat{\bf X})\). Moreover, for any \(\hat b\in{\rm Der}^\infty_\infty(\hat{\bf X})\)
and \(\hat h,\hat g\in\LIP_b(\hat\X,\hat\tau,\hat\sfd)\) we can compute
\[
\int\hat g\,\d{\bf L}(\hat h\hat b)=\int\hat g\circ\iota\,\d{\bf L}_f(\iota^*(\hat h\hat b))
=\int\hat g\circ\iota\,\d{\bf L}_f((\iota^*\hat h)(\iota^*\hat b))=\int(\hat g\hat h)\circ\iota\,\d{\bf L}_f(\iota^*\hat b)
=\int\hat g\hat h\,\d{\bf L}(\hat b),
\]
so that \({\bf L}(\hat h\hat b)=\hat h\,{\bf L}(\hat b)\) as a consequence of \eqref{eq:LIP_dense_L1}. All in all, we have shown
that \(\iota_*f\in{\rm BV}(\hat{\bf X})\) and that \({\bf L}_{\iota_*f}(\hat b)={\bf L}(\hat b)=\iota_\#({\bf L}_f(\iota^*\hat b))\)
for every \(\hat b\in{\rm Der}^\infty_\infty(\hat{\bf X})\). In particular, we can compute
\[\begin{split}
{\rm V}(\iota_*f)&=\sup\big\{\|{\bf L}_{\iota_*f}(\hat b)\|_{\rm TV}\;\big|\;\hat b\in\Der^\infty_\infty(\hat{\bf X}),\,
\|\hat b\|_{{\rm Der}^\infty(\hat{\bf X})}\leq 1\big\}\\
&=\sup\big\{\|{\bf L}_f(\iota^*\hat b)\|_{\rm TV}\;\big|\;\hat b\in\Der^\infty_\infty(\hat{\bf X}),\,
\|\hat b\|_{{\rm Der}^\infty(\hat{\bf X})}\leq 1\big\}\\
&=\sup\big\{\|{\bf L}_f(b)\|_{\rm TV}\;\big|\;b\in\Der^\infty_\infty({\bf X}),\,\|b\|_{{\rm Der}^\infty({\bf X})}\leq 1\big\}
={\rm V}(f),
\end{split}\]
where we use the identity \(\iota_\#|{\bf L}_f(\iota^*\hat b)|=|\iota_\#({\bf L}_f(\iota^*\hat b))|\), which can be
proved by taking the injectivity of \(\iota\) into account. This concludes the proof of the statement.
\end{proof}

We now combine the above statements to prove the main result of this section:
\begin{corollary}\label{cor:equiv_BV_on_Gelfand_cpt}
Let \({\bf X}=(\X,\sfd,\mm)\) be a m.m.s.\ such that \((\X,\sfd)\) is either complete or Radon. Then we have
\(\iota_*({\rm BV}({\bf X}))={\rm BV}(\hat{\bf X})={\rm BV}_*(\hat{\bf X})\) and \({\rm V}(f)={\rm V}(\iota_*f)={\rm V}_*(\iota_*f)\)
for every \(f\in{\rm BV}({\bf X})\).
\end{corollary}
\begin{proof}
Note that we have the following chains of inclusions and inequalities:
\[\begin{split}
{\rm BV}(\hat{\bf X})\overset{a)}{\subseteq}{\rm BV}_*(\hat{\bf X})\overset{b)}{\subseteq}\iota_*({\rm BV}_*({\bf X}))
\overset{c)}{=}\iota_*({\rm BV}({\bf X}))\overset{d)}{\subseteq}{\rm BV}(\hat{\bf X})&,\\
{\rm V}(\hat f)\overset{1)}{=}{\rm V}(\iota^*\hat f)\overset{2)}{=}{\rm V}_*(\iota^*\hat f)
\overset{3)}{\leq}{\rm V}_*(\hat f)\overset{4)}{\leq}{\rm V}(\hat f)&\quad\text{ for every }\hat f\in{\rm BV}(\hat{\bf X}).
\end{split}\]
Indeed, we have that: a) and 4) follow from Theorem \ref{thm:BV_in_BV*}; b) and 3) from Lemma \ref{lem:compos_BV*};
c) and 2) from Theorem \ref{thm:BV=BV*_mms}; d) and 1) from Proposition \ref{prop:compos_BV}. This gives the statement.
\end{proof}
\section{Predual of \texorpdfstring{\({\rm BV}_p({\bf X})\)}{BVp(X)} for \texorpdfstring{\(1<p<\infty\)}{1<p<infty}}

Let \({\bf X}=(\X,\tau,\sfd,\mm)\) be a locally-compact e.m.t.m.s. Observe that
\begin{equation}\label{eq:dual_Der_otimes_C}\begin{split}
\mathcal L\big(\overline{\rm Der}^\infty_\infty({\bf X});\mathcal M(\X,\tau)\big)
&\overset{\eqref{item:M(X)}}\approx_1\mathcal L\big(\overline{\rm Der}^\infty_\infty({\bf X});C_0(\X,\tau)^*\big)
\overset{\eqref{item:bdd_bilin}}\approx_1\mathcal B\big(\overline{\rm Der}^\infty_\infty({\bf X}),C_0(\X,\tau)\big)\\
&\overset{\eqref{item:proj_tensor_prod}}\approx_1\big(\overline{\rm Der}^\infty_\infty({\bf X})\hat\otimes_\pi C_0(\X,\tau)\big)^*.
\end{split}\end{equation}
We define the Banach space \(\mathbb W_q({\bf X})\) and its subspace \(\mathbb V_q({\bf X})\) as
\[\begin{split}
\mathbb W_q({\bf X})&\coloneqq L^q(\mm)\times_\infty\big(\overline{\rm Der}^\infty_\infty({\bf X})\hat\otimes_\pi C_0(\X,\tau)\big),\\
\mathbb V_q({\bf X})&\coloneqq\bigg\{\bigg(g,\sum_{i=1}^n b_i\otimes h_i\bigg)\in L^\infty(\mm)
\times({\rm Der}^\infty_\infty({\bf X})\otimes\LIP_b(\X,\tau,\sfd))\;\bigg|\;g=\div\bigg(\sum_{i=1}^n h_i b_i\bigg)\bigg\}.
\end{split}\]
\begin{remark}[Well-posedness of \(\mathbb V_q({\bf X})\)]\label{rmk:V_q_well-posed}{\rm
Let us check that the definition of \(\mathbb V_q({\bf X})\) is well posed.
Since \(\div\colon{\rm Der}^\infty_\infty({\bf X})\to L^\infty(\mm)\) is linear and
\({\rm Der}^\infty_\infty({\bf X})\times\LIP_b(\X,\tau,\sfd)\ni(b,h)\mapsto hb\in{\rm Der}^\infty_\infty({\bf X})\)
is bilinear, we have that
\({\rm Der}^\infty_\infty({\bf X})\times\LIP_b(\X,\tau,\sfd)\ni(b,h)\mapsto\delta(b,h)\coloneqq\div(hb)\in L^\infty(\mm)\)
is bilinear. Hence, the universal property of the tensor product of vector spaces
\(({\rm Der}^\infty_\infty({\bf X})\otimes\LIP_b(\X,\tau,\sfd),\otimes)\) ensures that there exists a unique
linear map \(\hat\div\colon{\rm Der}^\infty_\infty({\bf X})\otimes\LIP_b(\X,\tau,\sfd)\to L^\infty(\mm)\) such that
\[\begin{tikzcd}
{\rm Der}^\infty_\infty({\bf X})\times\LIP_b(\X,\tau,\sfd) \arrow[d,swap,"\otimes"]
\arrow[r,"\delta"] & L^\infty(\mm) \\
{\rm Der}^\infty_\infty({\bf X})\otimes\LIP_b(\X,\tau,\sfd) \arrow[ur,swap,"\hat\div"] &
\end{tikzcd}\]
is a commutative diagram. It follows that \(\div\big(\sum_{i=1}^n h_i b_i\big)=\hat\div(\alpha)\) is independent of the specific
way of writing \(\alpha=\sum_{i=1}^n b_i\otimes h_i\in{\rm Der}^\infty_\infty({\bf X})\otimes\LIP_b(\X,\tau,\sfd)\),
thus \(\mathbb V_q({\bf X})\) is well defined.
}\fr\end{remark}

The argument in Remark \ref{rmk:V_q_well-posed} shows also that \(\mathbb V_q({\bf X})\) is the zero set of the linear operator
\[
L^\infty(\mm)\times\big({\rm Der}^\infty_\infty({\bf X})\otimes\LIP_b(\X,\tau,\sfd)\big)\ni(g,\alpha)
\longmapsto g-\hat\div(\alpha)\in L^\infty(\mm),
\]
so that \(\mathbb V_q({\bf X})\) is a linear subspace of \(\mathbb W_q({\bf X})\) and its
closure \(\overline{\mathbb V}_q({\bf X})\) in \(\mathbb W_q({\bf X})\) is a Banach space.
\medskip

Now, observe that
\[\begin{split}
\mathbb W_q({\bf X})^*\overset{\eqref{item:ell_infty_product}}\approx_1
L^p(\mm)\times_1\big(\overline{\rm Der}^\infty_\infty({\bf X})\hat\otimes_\pi C_0(\X,\tau)\big)^*
\overset{\eqref{eq:dual_Der_otimes_C}}{\approx_1}L^p(\mm)\times_1\mathcal L\big(\overline{\rm Der}^\infty_\infty({\bf X});\mathcal M(\X,\tau)\big).
\end{split}\]
Keeping track of the various identifications, it can be readily checked that an isometric isomorphism of Banach spaces
\({\rm i}_{p,{\bf X}}\colon L^p(\mm)\times_1\mathcal L\big(\overline{\rm Der}^\infty_\infty({\bf X});\mathcal M(\X,\tau)\big)
\to\mathbb W_q({\bf X})^*\) is given by the unique bounded linear operator from
\(L^p(\mm)\times_1\mathcal L\big(\overline{\rm Der}^\infty_\infty({\bf X});\mathcal M(\X,\tau)\big)\)
to \(\mathbb W_q({\bf X})^*\) that verifies the identity
\begin{equation}\label{eq:main_dual_pairing}
\big\langle{\rm i}_{p,{\bf X}}(f,{\bf L}),(g,b\otimes h)\big\rangle=\int fg\,\d\mm+\int h\,\d{\bf L}(b)
\end{equation}
for every \((f,{\bf L},g,b,h)\in L^p(\mm)\times\mathcal L\big(\overline{\rm Der}^\infty_\infty({\bf X});\mathcal M(\X,\tau)\big)
\times L^q(\mm)\times{\rm Der}^\infty_\infty({\bf X})\times C_0(\X,\tau)\).
\begin{theorem}[Isometric predual of \({\rm BV}_p({\bf X})\)]\label{thm:predual_BV_p}
Let \({\bf X}=(\X,\tau,\sfd,\mm)\) be a locally-compact e.m.t.m.s. Let \(p,q\in(1,\infty)\) be
conjugate exponents. Let us define the operator \(\phi_{p,{\bf X}}\colon{\rm BV}_p({\bf X})\to\mathbb W_ q({\bf X})^*\) as
\[
\phi_{p,{\bf X}}(f)\coloneqq{\rm i}_{p,{\bf X}}(f,{\bf L}_f)\in\mathbb W_ q({\bf X})^*
\quad\text{ for every }f\in{\rm BV}_p({\bf X}).
\]
Then \(\phi_{p,{\bf X}}\) is a linear isometry and
\(\phi_{p,{\bf X}}({\rm BV}_p({\bf X}))=\overline{\mathbb V}_q({\bf X})^\perp_{\mathbb W_q({\bf X})}\).
In particular, we have
\[
{\rm BV}_p({\bf X})\approx_1(\mathbb W_q({\bf X})/\overline{\mathbb V}_q({\bf X}))^*.
\]
\end{theorem}
\begin{proof}
Given that \({\rm BV}_p({\bf X})\ni f\mapsto(f,{\bf L}_f)\in L^p(\mm)\times_1
\mathcal L\big(\overline{\rm Der}^\infty_\infty({\bf X});\mathcal M(\X,\tau)\big)\) and \({\rm i}_{p,{\bf X}}\) are
linear isometries, \(\phi_{p,{\bf X}}\) is a linear isometry. For any \(f\in{\rm BV}_p({\bf X})\)
and \(\big(g,\sum_{i=1}^n b_i\otimes h_i\big)\in\mathbb V_q({\bf X})\), one has
\[\begin{split}
\bigg\langle\phi_{p,{\bf X}}(f),\bigg(g,\sum_{i=1}^n b_i\otimes h_i\bigg)\bigg\rangle&\overset{\eqref{eq:main_dual_pairing}}=
\int fg\,\d\mm+\sum_{i=1}^n\int h_i\,\d{\bf L}_f(b_i)\\
&\overset{\phantom{\eqref{eq:main_dual_pairing}}}=\int f\,\div\bigg(\sum_{i=1}^n h_i b_i\bigg)\,\d\mm+\sum_{i=1}^n{\bf L}_f(h_i b_i)(\X)=0,
\end{split}\]
whence it follows that \(\phi_{p,{\bf X}}(f)\in\overline{\mathbb V}_q({\bf X})^\perp_{\mathbb W_q({\bf X})}\), so that the inclusion
\(\phi_{p,{\bf X}}({\rm BV}_p({\bf X}))\subseteq\overline{\mathbb V}_q({\bf X})^\perp_{\mathbb W_q({\bf X})}\) is proved.
To show the converse inclusion, fix any element \((f,{\bf L})\in L^p(\mm)\times_1\mathcal L\big(\overline{\rm Der}^\infty_\infty({\bf X});\mathcal M(\X,\tau)\big)\)
such that \({\rm i}_{p,{\bf X}}(f,{\bf L})\in\overline{\mathbb V}_q({\bf X})^\perp_{\mathbb W_q({\bf X})}\). We claim that \(f\in{\rm BV}_p({\bf X})\)
and \({\bf L}_f={\bf L}\). Fix \(b\in{\rm Der}^\infty_\infty({\bf X})\). Since \((\div(b),b\otimes\1_\X)\in\mathbb V_q({\bf X})\), we have that
\[
\int f\,\div(b)\,\d\mm+{\bf L}(b)(\X)=\langle{\rm i}_{p,{\bf X}}(f,{\bf L}),(\div(b),b\otimes\1_\X)\rangle=0,
\]
which proves that the operator \({\bf L}\) verifies \eqref{eq:def_BV_der_1}. Moreover, note that for any \(h,\psi\in\LIP_b(\X,\tau,\sfd)\) we have that
\((0,b\otimes(\psi h)-(hb)\otimes\psi)\in\mathbb V_q({\bf X})\), thus accordingly
\[
\int\psi\,\d\big(h{\bf L}(b)-{\bf L}(hb)\big)=\int\psi h\,\d{\bf L}(b)-\int\psi\,\d{\bf L}(hb)=\langle{\rm i}_{p,{\bf X}}(f,{\bf L}),(0,b\otimes(\psi h)-(hb)\otimes\psi)\rangle=0.
\]
Since \(\LIP_b(\X,\tau,\sfd)\) is dense in \(L^1(|h{\bf L}(b)-{\bf L}(hb)|)\), we deduce that \(h{\bf L}(b)={\bf L}(hb)\)
holds for every \(h\in\LIP_b(\X,\tau,\sfd)\), thus proving that \({\bf L}\) verifies \eqref{eq:def_BV_der_2}. All in all, we have shown that
\(f\in{\rm BV}_p({\bf X})\) and \({\bf L}_f={\bf L}\), as we claimed. It follows that \({\rm i}_{p,{\bf X}}(f,{\bf L})=\phi_{p,{\bf X}}(f)\in\phi_{p,{\bf X}}({\rm BV}_p({\bf X}))\).
Hence, also the inclusion \(\overline{\mathbb V}_q({\bf X})^\perp_{\mathbb W_q({\bf X})}\subseteq\phi_{p,{\bf X}}({\rm BV}_p({\bf X}))\) is proved.
Finally, the last part of the statement follows from \eqref{item:annihilator} and the fact that \({\rm BV}_p({\bf X})\approx_1\phi_{p,{\bf X}}({\rm BV}_p({\bf X}))\)
(since \(\phi_{p,{\bf X}}\) is a linear isometry).
\end{proof}
\begin{corollary}\label{cor:predual_BV_mms}
Let \({\bf X}=(\X,\sfd,\mm)\) be a m.m.s.\ such that \((\X,\sfd)\) is either complete or Radon.
Let \(p,q\in(1,\infty)\) be conjugate exponents. Then \({\rm BV}_p({\bf X})\) has an isometric predual. More precisely,
\[
{\rm BV}_p({\bf X})\approx_1(\mathbb W_q(\hat{\bf X})/\overline{\mathbb V}_q(\hat{\bf X}))^*.
\]
\end{corollary}
\begin{proof}
Note that \({\rm BV}_p({\bf X})\approx_1{\rm BV}_p(\hat{\bf X})\approx_1(\mathbb W_q(\hat{\bf X})/\overline{\mathbb V}_q(\hat{\bf X}))^*\)
by Corollary \ref{cor:equiv_BV_on_Gelfand_cpt} and Theorem \ref{thm:predual_BV_p}.
\end{proof}

In the next result, we characterise the weakly\(^*\) converging bounded nets in the space \({\rm BV}_p({\bf X})\).
\begin{proposition}
Let \({\bf X}=(\X,\tau,\sfd,\mm)\) be a locally-compact e.m.t.m.s. Let \(p,q\in(1,\infty)\) be conjugate
exponents. Fix a bounded net \((f_i)_{i\in I}\) in \({\rm BV}_p({\bf X})\) and a function \(f\in{\rm BV}_p({\bf X})\).
Then we have \(\lim_{i\in I}f_i=f\) in the weak\(^*\) topology of
\({\rm BV}_p({\bf X})\approx_1(\mathbb W_q({\bf X})/\overline{\mathbb V}_q({\bf X}))^*\) if and only if
\begin{align}
\label{eq:weak_star_BV_1}
\lim_{i\in I}f_i=f&\quad\text{ in the weak topology of }L^p(\mm),\\
\label{eq:weak_star_BV_2}
\lim_{i\in I}{\bf L}_{f_i}(b)={\bf L}_f(b)&\quad\text{ in the weak\(^*\) topology of }\mathcal M(\X,\tau),
\text{ for every }b\in{\rm Der}^\infty_\infty({\bf X}).
\end{align}
\end{proposition}
\begin{proof}
Assume \(\lim_{i\in I}f_i=f\) weakly\(^*\) in \({\rm BV}_p({\bf X})\). Taking \eqref{eq:main_dual_pairing} into account,
we deduce that
\[\begin{split}
&\lim_{i\in I}\int f_i g\,\d\mm=\lim_{i\in I}\langle\phi_{p,{\bf X}}(f_i),(g,0)\rangle=\langle\phi_{p,{\bf X}}(f),(g,0)\rangle
=\int fg\,\d\mm,\\
&\lim_{i\in I}\int h\,\d{\bf L}_{f_i}(b)=\lim_{i\in I}\langle\phi_{p,{\bf X}}(f_i),(0,b\otimes h)\rangle
=\langle\phi_{p,{\bf X}}(f),(0,b\otimes h)\rangle=\int h\,\d{\bf L}_f(b)
\end{split}\]
for every \(g\in L^q(\mm)\), \(b\in{\rm Der}^\infty_\infty({\bf X})\) and \(h\in C_0(\X,\tau)\), thus proving
\eqref{eq:weak_star_BV_1} and \eqref{eq:weak_star_BV_2}, respectively.

Conversely, assume \eqref{eq:weak_star_BV_1} and \eqref{eq:weak_star_BV_2} hold. Fix \(\big(g,\sum_{j=1}^n b_j\otimes h_j\big)
\in L^q(\mm)\times({\rm Der}^\infty_\infty({\bf X})\otimes C_0(\X,\tau))\). Then we can compute
\begin{equation}\label{eq:equiv_weak_star_BV}\begin{split}
\lim_{i\in I}\bigg\langle\phi_{p,{\bf X}}(f_i),\bigg(g,\sum_{j=1}^n b_j\otimes h_j\bigg)\bigg\rangle
&=\lim_{i\in I}\bigg(\int f_i g\,\d\mm+\sum_{j=1}^n\int h_j\,\d{\bf L}_{f_i}(b_j)\bigg)\\
&=\int fg\,\d\mm+\sum_{j=1}^n\int h_j\,\d{\bf L}_f(b_j)\\
&=\bigg\langle\phi_{p,{\bf X}}(f),\bigg(g,\sum_{j=1}^n b_j\otimes h_j\bigg)\bigg\rangle.
\end{split}\end{equation}
Now, fix any \((g,\alpha)\in\mathbb W_q({\bf X})\) and \(\varepsilon>0\). Since \({\rm Der}^\infty_\infty({\bf X})\otimes C_0(\X,\tau)\)
is dense in \(\overline{\rm Der}^\infty_\infty({\bf X})\hat\otimes_\pi C_0(\X,\tau)\), we can find a tensor
\(\beta=\sum_{j=1}^n b_j\otimes h_j\in{\rm Der}^\infty_\infty({\bf X})\otimes C_0(\X,\tau)\) such that
\(\|(g,\alpha)-(g,\beta)\|_{\mathbb W_q({\bf X})}\leq\varepsilon\). Therefore, letting
\(M\coloneqq\|f\|_{{\rm BV}_p({\bf X})}+\sup_{i\in I}\|f_i\|_{{\rm BV}_p({\bf X})}<+\infty\), we deduce from \eqref{eq:equiv_weak_star_BV} that
\[
\varlimsup_{i\in I}\big|\langle\phi_{p,{\bf X}}(f_i),(g,\alpha)\rangle-\langle\phi_{p,{\bf X}}(f),(g,\alpha)\rangle\big|
\leq M\varepsilon+\lim_{i\in I}\big|\langle\phi_{p,{\bf X}}(f_i),(g,\beta)\rangle-\langle\phi_{p,{\bf X}}(f),(g,\beta)\rangle\big|
=M\varepsilon,
\]
so that \(\lim_{i\in I}\langle\phi_{p,{\bf X}}(f_i),(g,\alpha)\rangle=\langle\phi_{p,{\bf X}}(f),(g,\alpha)\rangle\)
by the arbitrariness of \(\varepsilon>0\). This proves that \(\lim_{i\in I}f_i=f\) in the weak\(^*\) topology
of \({\rm BV}_p({\bf X})\).
\end{proof}
\section{Predual of \texorpdfstring{\({\rm BV}({\bf X})\)}{BV(X)} on bounded PI spaces}\label{s:predual_BV_PI}
In general, Theorem \ref{thm:predual_BV_p} (and Corollary \ref{cor:predual_BV_mms}) cannot be generalized to the limit case \(p=1\).
Indeed, in the following example we construct a metric measure space \({\bf X}\) whose BV space \({\rm BV}({\bf X})\) does not have
an isometric predual (and, in fact, not even a predual).
\begin{example}\label{ex:no_predual}{\rm
Let us equip \([0,1]\) with the snowflake distance \(\sfd(x,y)\coloneqq\sqrt{|x-y|}\) and the Lebesgue measure
\(\mm\coloneqq\mathscr L^1|_{[0,1]}\). Note that \({\bf X}=([0,1],\sfd,\mm)\) is a complete m.m.s. We claim that
\[
{\rm BV}({\bf X})=L^1(0,1),\qquad{\rm V}(f)=0\quad\text{ for every }f\in{\rm BV}({\bf X}).
\]
To prove it, fix any function \(f\in L^1(0,1)\). Take a sequence \((f_n)_n\subseteq\LIP([0,1],\sfd_{\rm Eu})\) such that
\(f_n\to f\) in \(L^1(0,1)\), where \(\sfd_{\rm Eu}\) denotes the Euclidean distance \(\sfd_{\rm Eu}(x,y)\coloneqq|x-y|\).
Since \(\sfd_{\rm Eu}\leq\sfd\), we have \((f_n)_n\subseteq\LIP([0,1],\sfd)\). Note also that for every \(n\in\N\) and \(x\in[0,1]\) we have
\[\begin{split}
\lip_a^\sfd(f_n)(x)&=\lim_{r\searrow 0}\sup\bigg\{\frac{|f_n(y)-f_n(z)|}{\sfd(y,z)}\;\bigg|\;y,z\in[0,1],\,y\neq z,\,\sfd(x,y),\sfd(x,z)<r\bigg\}\\
&=\lim_{r\searrow 0}\sup\bigg\{\sfd(y,z)\frac{|f_n(y)-f_n(z)|}{\sfd_{\rm Eu}(y,z)}\;\bigg|\;y,z\in[0,1],\,y\neq z,\,\sfd(x,y),\sfd(x,z)<r\bigg\}\\
&\leq\Lip^{\sfd_{\rm Eu}}(f_n)\lim_{r\searrow 0}\sup\big\{\sfd(y,z)\;\big|\;y,z\in[0,1],\,y\neq z,\,\sfd(x,y),\sfd(x,z)<r\big\}=0,
\end{split}\]
thus \(f\in{\rm BV}_*({\bf X})\) and \({\rm V}_*(f)\leq\lim_n\int\lip_a^\sfd(f_n)\,\d\mm=0\). Recalling Theorem \ref{thm:BV=BV*_mms}, the claim is proved.
Therefore, the space \({\rm BV}({\bf X})\approx_1 L^1(0,1)\) does not have a predual by \cite[Theorem 6.3.7]{Albiac-Kalton}. 
\fr}\end{example}

A characteristic feature of the example we presented in Example \ref{ex:no_predual} is that the total variation of a BV function
does not control the behaviour of the function itself. A distinguished class of metric measure spaces where this phenomenon does not occur
is that of the so-called \emph{PI spaces}, whose definition we report below. We refer to \cite{Hei:Kos:98,hajlasz2000sobolev,Bj:Bj:11,HKST:15}
for a detailed account of this topic.
\begin{definition}[PI space]\label{def:PI_space}
A complete m.m.s.\ \({\bf X}=(\X,\sfd,\mm)\) is said to be a \textbf{PI space} provided:
\begin{enumerate}[(i)]
\item \textsc{Doubling condition.} There exists a constant \(C_D({\bf X})>0\) such that
\[
0<\mm(B_{2r}(x))\leq C_D({\bf X})\mm(B_r(x))<+\infty\quad\text{ for every }x\in\X\text{ and }r>0.
\]
\item \textsc{Weak \((1,1)\)-Poincar\'{e} inequality.} There exist constants \(C_P({\bf X})>0\) and \(\lambda({\bf X})\geq 1\) such that for every
Lipschitz function \(f\colon\X\to\R\) we have
\[
\fint_{B_r(x)}\bigg|f-\fint_{B_r(x)}f\,\d\mm\bigg|\,\d\mm\leq C_P({\bf X})r\fint_{B_{\lambda({\bf X})r}(x)}\lip_a(f)\,\d\mm\quad\text{ for every }x\in\X\text{ and }r>0.
\]
\end{enumerate}
\end{definition}

Our next goal is to show that, on a bounded PI space, \({\rm BV}({\bf X})={\rm BV}_p({\bf X})\) with equivalent norms for all \(p>1\)
sufficiently large, whence it follows (taking Corollary \ref{cor:predual_BV_mms} into account) that in this case the space \({\rm BV}({\bf X})\)
has a predual; see Theorem \ref{thm:predual_BV_PI}. To achieve this goal, we first need to discuss some preliminary notions and results.
\medskip

A Borel subset \(E\subseteq\X\) of a m.m.s.\ \({\bf X}=(\X,\sfd,\mm)\) is said to be a \textbf{set of finite perimeter} provided its characteristic function
\(\1_E\) belongs to \({\rm BV}({\bf X})\). The \textbf{perimeter} of \(E\) is defined as
\[
{\rm P}(E)\coloneqq{\rm V}(\1_E).
\]
The following result, which correlates the total variation of a BV function with the perimeters of its superlevel sets, is valid on
(complete separable) m.m.s.\ and was proved in \cite[Proposition 4.2]{Mir:03}.
\begin{theorem}[Coarea formula]\label{thm:coarea}
Let \({\bf X}=(\X,\sfd,\mm)\) be a m.m.s.\ such that \((\X,\sfd)\) is complete and separable. Let \(f\in{\rm BV}({\bf X})\) be given.
Then \(\{f>t\}\) is a set of finite perimeter for a.e.\ \(t\in\R\), the function \(\R\ni t\mapsto{\rm P}(\{f>t\})\) is Lebesgue measurable, and we have
\[
{\rm V}(f)=\int_\R{\rm P}(\{f>t\})\,\d t.
\]
\end{theorem}

Another ingredient we need is the following isoperimetric inequality, which holds on PI spaces:
\begin{theorem}[Global isoperimetric inequality on bounded PI spaces]\label{thm:glob_isoper_ineq}
Let \({\bf X}=(\X,\sfd,\mm)\) be a bounded PI space. Fix any \(\alpha\in(1,+\infty)\) with \(\alpha\geq\log_2(C_D({\bf X}))\).
Then there exists a constant \(C_I({\bf X},\alpha)>0\), depending only on \(C_D({\bf X})\), \(C_P({\bf X})\),
\(\lambda({\bf X})\), \({\rm diam}(\X)\), \(\mm(\X)\) and \(\alpha\), such that
\begin{equation}\label{eq:glob_isoper_ineq}
\min\{\mm(E),\mm(\X\setminus E)\}\leq C_I({\bf X},\alpha){\rm P}(E)^{\frac{\alpha}{\alpha-1}}
\quad\text{ for every set of finite perimeter }E\subseteq\X.
\end{equation}
\end{theorem}
\begin{proof}
We know from \cite[Theorem 4.5]{Mir:03} that the following \emph{relative isoperimetric inequality} holds: for some constant
\(\tilde C_I({\bf X},\alpha)>0\), depending only on \(C_D({\bf X})\), \(C_P({\bf X})\), \(\lambda({\bf X})\) and \(\alpha\), we have
\begin{equation}\label{eq:rel_isoper_ineq}
\min\{\mm(B_r(x)\cap E),\mm(B_r(x)\setminus E)\}\leq\tilde C_I({\bf X},\alpha)
\bigg(\frac{r^\alpha}{\mm(B_r(x))}\bigg)^{\frac{1}{\alpha-1}}{\rm P}(E)^{\frac{\alpha}{\alpha-1}}
\end{equation}
whenever \(E\subseteq\X\) is a set of finite perimeter, \(x\in\X\) and \(r>0\). Since \(B_{{\rm diam}(\X)}(x)=\X\),
by choosing \(r\coloneqq{\rm diam}(\X)\) in \eqref{eq:rel_isoper_ineq}, we obtain that \eqref{eq:glob_isoper_ineq}
holds with \(C_I({\bf X},\alpha)\coloneqq\tilde C_I({\bf X},\alpha)\big(\frac{{\rm diam}(\X)^\alpha}{\mm(\X)}\big)^{\frac{1}{\alpha-1}}\).
\end{proof}

For the sake of conciseness, we introduce the shorthand notation
\[
\alpha_*\coloneqq\frac{\alpha}{\alpha-1}\in(1,+\infty)\quad\text{ for every }\alpha\in(1,+\infty).
\]
By adapting the proof of the embedding theorem for Euclidean BV functions into \(L^{n/(n-1)}(\R^n)\)
(see e.g.\ \cite[Theorem 3.47]{AmbFusPal00}), we obtain the following result:
\begin{theorem}[Embedding theorem]\label{thm:embedding_thm}
Let \({\bf X}=(\X,\sfd,\mm)\) be a bounded PI space and \(f\in{\rm BV}({\bf X})\). Fix \(\alpha\in(1,+\infty)\) with
\(\alpha\geq\log_2(C_D({\bf X}))\). Then there exists \(c_f\in\R\) such that \(|c_f|\leq\frac{2\|f\|_{L^1(\mm)}}{\mm(\X)}\) and
\begin{equation}\label{eq:embedding_ineq}
\|f-c_f||_{L^{\alpha_*}(\mm)}\leq 2 C_I({\bf X},\alpha)^{\frac{1}{\alpha_*}}{\rm V}(f).
\end{equation}
\end{theorem}
\begin{proof}
Since the function \(\R\ni t\mapsto\mm(\{f>t\})\) is non-increasing, we can find \(c_f\in\R\) such that
\begin{align}
\label{eq:c_f_prop_1}
\mm(\{f\leq t\})\leq\mm(\{f>t\})&\quad\text{ for every }t\in(-\infty,c_f),\\
\label{eq:c_f_prop_2}
\mm(\{f>t\})\leq\mm(\{f\leq t\})&\quad\text{ for every }t\in[c_f,+\infty).
\end{align}
Let us check that \(|c_f|\leq\frac{2\|f\|_{L^1(\mm)}}{\mm(\X)}\). If \(c_f>0\), then \eqref{eq:c_f_prop_1} and Chebyshev's inequality yield
\[
\mm(\X)\leq 2\,\mm(\{f>c_f-\varepsilon\})\leq\frac{2}{c_f-\varepsilon}\int_{\{f\geq c_f-\varepsilon\}}f\,\d\mm\leq
\frac{2}{c_f-\varepsilon}\int|f|\,\d\mm\quad\text{ for every }\varepsilon\in(0,c_f),
\]
whence it follows that \(c_f\leq\frac{2\|f\|_{L^1(\mm)}}{\mm(\X)}\). If \(c_f<0\), then \eqref{eq:c_f_prop_2} and Chebyshev's inequality yield
\[
\mm(\X)\leq 2\,\mm(\{f\leq c_f\})=2\,\mm(\{-f\geq -c_f\})\leq\frac{2}{-c_f}\int_{\{-f\leq -c_f\}}-f\,\d\mm
\leq\frac{2}{-c_f}\int|f|\,\d\mm,
\]
so that \(-c_f\leq\frac{2\|f\|_{L^1(\mm)}}{\mm(\X)}\). All in all, we have shown that \(|c_f|\leq\frac{2\|f\|_{L^1(\mm)}}{\mm(\X)}\).
Let us now pass to the verification of \eqref{eq:embedding_ineq}. Denote \(f_+\coloneqq(f-c_f)^+\) and \(f_-\coloneqq(f-c_f)^-\).
Then we can estimate
\[\begin{split}
\int f_\pm^{\alpha_*}\,\d\mm&\overset{{\rm a)}}{=}\int_0^{+\infty}\mm(\{f_\pm^{\alpha_*}>r\})\,\d r
\overset{{\rm b)}}{=}\alpha_*\int_0^{+\infty}\mm(\{f_\pm>t\})t^{\alpha_*-1}\,\d t\\
&\overset{{\rm c)}}{\leq}\bigg(\int_0^{+\infty}\mm(\{f_\pm>t\})^{\frac{1}{\alpha_*}}\,\d t\bigg)^{\alpha_*}
\overset{{\rm d)}}{\leq}C_I({\bf X},\alpha)\bigg(\int_0^{+\infty}{\rm P}(\{f_\pm>t\})\,\d t\bigg)^{\alpha_*}\\
&\overset{{\rm e)}}{=}C_I({\bf X},\alpha){\rm V}(f_\pm)^{\alpha_*}
\overset{{\rm f)}}{\leq}C_I({\bf X},\alpha){\rm V}(f)^{\alpha_*},
\end{split}\]
where: for \({\rm a)}\), we use the layer cake representation formula; for \({\rm b)}\), we apply the change of variables
\(t=r^{\frac{1}{\alpha_*}}\); for \({\rm c)}\), we apply \cite[Lemma 3.48]{AmbFusPal00}; for \({\rm d)}\), we use the
global isoperimetric inequality (Theorem \ref{thm:glob_isoper_ineq}) and the fact that \(\mm(\{f_\pm>t\})\leq\mm(\{f_\pm\leq t\})\)
for a.e.\ \(t>0\) (by \eqref{eq:c_f_prop_1} and \eqref{eq:c_f_prop_2}); for \({\rm e)}\), we apply \eqref{eq:chain_rule_BV}
and the coarea formula (Theorem \ref{thm:coarea}); for \({\rm f)}\), we use \eqref{eq:const_BV} and \eqref{eq:chain_rule_BV}
(together with Corollary \ref{cor:predual_BV_mms}). Given that \(f-c_f=f_+ - f_-\), we conclude that
\[
\|f-c_f\|_{L^{\alpha_*}(\mm)}\leq\|f_+\|_{L^{\alpha_*}(\mm)}+\|f_-\|_{L^{\alpha_*}(\mm)}\leq 2C_I({\bf X},\alpha)^{\frac{1}{\alpha_*}}{\rm V}(f),
\]
thus proving \eqref{eq:embedding_ineq}.
\end{proof}

As a consequence of Theorem \ref{thm:embedding_thm}, we obtain that \({\rm BV}({\bf X})\approx{\rm BV}_{\alpha_*}({\bf X})\)
for all large \(\alpha>1\):
\begin{corollary}\label{cor:PI_BV=BV_alpha_star}
Let \({\bf X}=(\X,\sfd,\mm)\) be a bounded PI space. Fix \(\alpha\in(1,+\infty)\) with \(\alpha\geq\log_2(C_D({\bf X}))\).
Then we have \({\rm BV}({\bf X})={\rm BV}_{\alpha_*}({\bf X})\). Moreover, for any \(f\in{\rm BV}({\bf X})\) we have
\begin{equation}\label{eq:equiv_BV_BV_alpha*}
\frac{1}{\max\{\mm(\X)^{\frac{1}{\alpha}},1\}}\|f\|_{{\rm BV}({\bf X})}\leq\|f\|_{{\rm BV}_{\alpha_*}({\bf X})}
\leq\max\{2C_I({\bf X},\alpha)^{\frac{1}{\alpha_*}}+1,2\mm(\X)^{-\frac{1}{\alpha}}\}\|f\|_{{\rm BV}({\bf X})}.
\end{equation}
\end{corollary}
\begin{proof}
To obtain the first inequality in \eqref{eq:equiv_BV_BV_alpha*}, observe that
\(\|f\|_{L^1(\mm)}\leq\mm(\X)^{\frac{1}{\alpha}}\|f\|_{L^{\alpha_*}(\mm)}\) by H\"{o}lder's inequality.
To obtain the second inequality in \eqref{eq:equiv_BV_BV_alpha*}, observe that
\[
\|f\|_{L^{\alpha_*}(\mm)}\leq\|f-c_f\|_{L^{\alpha_*}(\mm)}+|c_f|\mm(\X)^{\frac{1}{\alpha_*}}
\leq 2 C_I({\bf X},\alpha)^{\frac{1}{\alpha_*}}{\rm V}(f)+2\mm(\X)^{-\frac{1}{\alpha}}\|f\|_{L^1(\mm)}
\]
thanks to Theorem \ref{thm:embedding_thm}.
\end{proof}

Finally, we are now in a position to prove the last result of this paper:
\begin{theorem}[Predual of \({\rm BV}({\bf X})\) when \({\bf X}\) is a bounded PI space]\label{thm:predual_BV_PI}
Let \({\bf X}=(\X,\sfd,\mm)\) be a bounded PI space. Then \({\rm BV}({\bf X})\)
has predual \(\mathbb W_\alpha({\bf X})/\overline{\mathbb V}_\alpha({\bf X})\) for every \(\alpha\in(1,+\infty)\)
with \(\alpha\geq\log_2(C_D({\bf X}))\).
\end{theorem}
\begin{proof}
Since \(\alpha_*\) and \(\alpha\) are conjugate exponents, we know that
\({\rm BV}_{\alpha_*}({\bf X})\approx_1(\mathbb W_\alpha({\bf X})/\overline{\mathbb V}_\alpha({\bf X}))^*\)
by Theorem \ref{thm:predual_BV_p}. Given that \({\rm BV}({\bf X})\approx{\rm BV}_{\alpha_*}({\bf X})\) by
Corollary \ref{cor:PI_BV=BV_alpha_star}, the statement follows.
\end{proof}
\end{document}